\documentclass[12pt,english]{amsart}
\usepackage[utf8]{inputenc}

\usepackage{amsbsy}  
\usepackage{amsmath}
\usepackage{amsfonts}
\usepackage{amssymb}
\usepackage{amsthm}

\usepackage[T1]{fontenc}
\usepackage[english]{babel}

\usepackage{ae,aecompl}
\usepackage{enumerate}

\usepackage[dvips]{graphicx}
\usepackage{a4}
\usepackage{epsfig}
\usepackage{t1enc}
\usepackage{subfigure}
\usepackage{fancyhdr}
\usepackage{ushort}

\parskip\bigskipamount

\numberwithin{equation}{section}

\newtheorem{theorem}[equation]{Theorem}
\newtheorem{lemma}[equation]{Lemma}

\newtheorem{corollary}[equation]{Corollary}

\theoremstyle{definition}
\newtheorem{definition}[equation]{Definition}

\theoremstyle{remark}

\newcommand{\R}{\mathbb{R}}

\newcommand{\Z}{\mathbb{Z}}
\newcommand{\N}{\mathbb{N}}
\newcommand{\D}{\mathcal{D}}
\newcommand{\C}{\mathcal{C}}

\newcommand{\eps}{\varepsilon}
\newcommand{\cl}{\overline}
\newcommand{\diam}{\mathrm{diam}}
\newcommand{\dmu}{\, \mathrm{d} \mu}

\newcommand{\Qka}{Q^k_\alpha}
\newcommand{\Qkb}{Q^k_\beta}
\newcommand{\Qlb}{Q^l_\beta}
\newcommand{\zka}{z^k_\alpha}
\newcommand{\zkb}{z^k_\beta}
\newcommand{\zlb}{z^l_\beta}

\def\pre{~\lower4pt\hbox{$\stackrel{\displaystyle
\succ}{\displaystyle \sim}$}~}
\def\prei{~\lower4pt\hbox{${\stackrel{\displaystyle
\succ}{\displaystyle \sim}}_i$}~}
\def\prej{~\lower4pt\hbox{${\stackrel{\displaystyle
\succ}{\displaystyle \sim}}_j$}~}
\def\preI{~\lower4pt\hbox{${\stackrel{\displaystyle
\succ}{\displaystyle \sim}}_I$}~}
\def\pres{~\lower4pt\hbox{${\stackrel{\displaystyle
\succ}{\displaystyle \sim}}^*$}~}

\begin{document}

\title{Dyadic cubes in spaces of homogeneous type}
\date{}
\author{Janne Korvenpää}
\maketitle
\begin{center}
Department of Mathematics and Systems Analysis, Aalto University, \\ Espoo, Finland \\
janne.korvenpaa@aalto.fi
\end{center}

%\clearpage
\section{Introduction}

The system of Euclidean dyadic cubes used in harmonic analysis can be generalized to more abstract spaces. In this text, we present Michael Christ's construction of dyadic cubes in a space of homogeneous type, i.e. a quasi-metric space with a doubling measure. The essential properties of the dyadic cubes are that they form a tree structure such that any two of them are either disjoint or one is contained in the other, and that each generation of cubes covers the whole space excluding a possible set of measure zero. In addition, dyadic cubes are not too far away from balls in the sense that they are bounded by balls of the same magnitude from inside and outside.

Christ's paper [\ref{christ}] is the only known piece of literature in which the construction of dyadic cubes is properly presented, and even it lacks several relevant steps and arguments in its proofs. This text follows the construction considered in Christ's paper without bringing out anything totally new. Instead, we have been aiming at presenting the construction more clearly, comprehensively and transparently. Especially, more steps and arguments have been written out in proofs, and the assumptions of lemmas are clearly stated. In addition, some new notation has been used to clarify the deductions made in proofs.

\section{Spaces of homogeneous type}

In this section, we introduce the space we are considering, and some tools needed for defining dyadic cubes. First, we need a quasi-metric.
\begin{definition}[Quasi-metric] \label{quasi}
A quasi-metric in the set $X$ is a function $\rho: X \times X \to [0,\infty)$ satisfying the following conditions for all $x,y,z \in X$:
\begin{eqnarray}
\label{quasizero} \rho(x,y) & = & 0 \text{ if and only if } x = y, \\
\label{quasisymmetricity} \rho(x,y) & = & \rho(y,x), \\
\label{quasitriangle} \rho(x,z) & \leq & A_0 \rho(x,y) + A_0 \rho(y,z),
\end{eqnarray}
where $A_0 \geq 1$ is a constant that does not depend on $x,y$ and $z$. 
\end{definition}
Like a normal metric, quasi-metric defines balls, diameters of subsets of $X$ and distances between subsets:
\begin{eqnarray*}
B(x,r) & := & \{y \in X : \rho(x,y) < r\}, \quad x \in X, r > 0, \\
\diam(A) & := & \sup_{x,y \in A} \rho(x,y), \quad A \subset X, \\
\rho(x,A) & := & \inf_{y \in A} \rho(x,y), \quad x \in X, A \subset X, \\
\quad \rho(A,B) & := & \inf_{x \in A, y \in B} \rho(x,y), \quad A,B \subset X.
\end{eqnarray*}
However, quasi-metric balls are not necessarily open sets.

In addition to a quasi-metric, we need a doubling measure that is consistent with the chosen quasi-metric.
\begin{definition}[Doubling measure] \label{doublingmeasure}
A doubling measure in the space $(X,\rho)$ is a Borel regular measure $\mu$ such that the balls of $(X,\rho)$ are $\mu$-measurable sets and the following conditions hold for all $x \in X$, $r > 0$:
\begin{equation} \label{doublingfinite}
0 < \mu(B(x,r)) < \infty,
\end{equation}
\begin{equation} \label{doubling}
\mu(B(x,2r)) \leq A_1 \mu(B(x,r)),
\end{equation}
where $A_1 \geq 1$ is a constant that does not depend on $x$ and $r$.
\end{definition}

Now we can define a space of homogeneous type.
\begin{definition}[Space of homogeneous type] \label{homogeneous}
A space of homogeneous type is a triple $(X,\rho,\mu)$, where $X$ is a non-empty set, $\rho$ a quasi-metric in $X$ such that all balls in the space $(X,\rho)$ are open sets, and $\mu$ a doubling measure in the space $(X,\rho)$.
\end{definition}

Lebesgue's theorem on differentiation holds in spaces of homogeneous type.
\begin{theorem}[Lebesgue's theorem on differentiation] \label{lebesgue}
Let $f$ be locally integrable function in the space of homogeneous type $(X,\rho,\mu)$. Then
\begin{equation*}
\lim_{r \to 0} \frac{1}{\mu(B(x,r))} \int_{B(x,r)} f \dmu = f(x)
\end{equation*}
for almost every $x \in X$.
\end{theorem}

We also need a property that a disperse subset in a space of homogeneous type cannot contain arbitralily many points in a ball. In the literature, such property is known as a finite Assouad dimension of the quasi-metric space $(X,\rho)$.
\begin{lemma} \label{assouad}
In the space of homogeneous type $(X,\rho,\mu)$, for each $K > 0$ there exists $N \in \N$ such that for each $x \in X$, for each $r > 0$ every set of the form
\begin{equation*}
A = \{ z_1, z_2, z_3, \ldots \, : \rho(z_i,z_j) \geq r, \text{ when } i \neq j \} \subset X
\end{equation*} 
contains at most $N$ points in the ball $B(x,Kr)$.
\end{lemma}

\section{Dyadic cubes}
Let $(X,\rho,\mu)$ be a space of homogeneous type with a quasi-triangle inequality constant $A_0$ and a doubling constant $A_1$.

We start constructing dyadic cubes by fixing a reference point in each set corresponding a dyadic cube. Let $\delta \in (0,1)$ be a small parameter to be determined later, and for each $k \in \Z$, fix a maximal set $Z_k \subset X$ such that
\begin{equation} \label{zka}
\rho(\zka,z^k_\beta) \geq \delta^k \quad \text{for every } \zka, \zkb \in Z_k, \zka \neq z^k_\beta.
\end{equation}
In this context, maximality means that no new points of the set $X$ can be added to the set $Z_k$ such that (\ref{zka}) remains valid. Thus, a maximal set is not unique here, but it exists. The set $Z_k$ can be finite or countably infinite depending on the space $(X,\rho)$. Anyway, $Z_k$ is non-empty since $X \neq \varnothing$.

The index $k$ determines the generation of a dyadic cube, and $\zka \in Z_k$ can be seen as the center of the corresponding dyadic cube $\Qka$ which will be defined further on. The parameter $\delta$, instead, determines the minimum distance between two centers in the same generation and the scaling between consecutive generations.

We label the points in the sets $Z_k$, $k \in \Z$, with an index set $I_k$ such that
\begin{equation}
\alpha \in I_k \text{ if and only if } \zka \in Z_k.
\end{equation}
By maximality of the sets $Z_k$, it holds
\begin{equation} \label{zkamaximality}
\text{for each } x \in X, \text{ for each } k \in \Z \text{ there exists } \alpha \in I_k \text{ such that } \rho(x,\zka) < \delta^k.
\end{equation}

When the reference points are fixed, they help us to construct a partial order in the set of index pairs formed by the reference points. Such partial order is needed to define the dyadic cubes.
\begin{lemma} \label{potree}
There exists a partial order $\preceq$ in the set $\{(k,\alpha) : k \in \Z, \alpha \in I_k\}$ satisfying the following conditions.
\begin{enumerate}[(a)]
\item \label{pogeneration} If $(k,\alpha) \preceq (l,\beta)$, then $k \geq l$.
\item \label{poanchestor} For each $(k, \alpha)$ and $l \leq k$ there exists a unique $\beta \in I_l$ such that $(k,\alpha) \preceq (l,\beta)$.
\item \label{pochildren} If $(k,\alpha) \preceq (k - 1, \beta)$, then $\rho(\zka, z^{k-1}_\beta) < \delta^{k-1}$.
\item \label{poparent} If $\displaystyle{\rho(\zka, z^{k-1}_\beta) < \frac{1}{2 A_0} \delta^{k-1}}$, then $(k,\alpha) \preceq (k - 1, \beta)$.
\end{enumerate}
\end{lemma}
\begin{proof}
By (\ref{zkamaximality}), for each pair $(k,\alpha)$ there exists at least one $\beta \in I_{k-1}$ such that $\rho(\zka,z^{k-1}_\beta) < \delta^{k-1}$. On the other hand, let us show that for the same pair there exists at most one $\beta \in I_{k-1}$ such that $\rho(\zka,z^{k-1}_\beta) < \frac{1}{2 A_0} \delta^{k-1}$: if $z^{k-1}_\gamma$ is another such point, then by the triangle inequality (\ref{quasitriangle})
\begin{eqnarray*}
\rho(z^{k-1}_\beta, z^{k-1}_\gamma) & \leq & A_0 \rho(z^{k-1}_\beta, \zka) + A_0 \rho(\zka, z^{k-1}_\gamma) \\
& < & A_0 \frac{1}{2 A_0} \delta^{k-1} + A_0 \frac{1}{2 A_0} \delta^{k-1} \\
& = & \delta^{k-1},
\end{eqnarray*}
contradicting (\ref{zka}).

The partial order $\preceq$ is constructed in the following way: for each $(k,\alpha)$ check, whether there exists an index $\beta \in I_{k-1}$ such that  $\rho(\zka,z^{k-1}_\beta) < \frac{1}{2 A_0} \delta^{k-1}$. If so, set $(k,\alpha) \preceq (k-1,\beta)$ and $(k,\alpha) \npreceq (k-1,\gamma)$ for all other $\gamma \in I_{k-1}$. If such $\beta$ does not exist, then select some $\beta \in I_{k-1}$ for which $\rho(\zka,z^{k-1}_\beta) < \delta^{k-1}$, and set $(k,\alpha) \preceq (k-1,\beta)$ and $(k,\alpha) \npreceq (k-1,\gamma)$ for all other $\gamma \in I_{k-1}$.

Finally, extend $\preceq$ by reflexivity, i.e. set $(k,\alpha) \preceq (k,\alpha)$ for each $(k,\alpha)$, and by transitivity, i.e. if $(k,\alpha) \preceq (l,\beta)$ and $(l,\beta) \preceq (m,\gamma)$, then set $(k,\alpha) \preceq (m,\gamma)$. Then $\preceq$ becomes a partial order since the last property, antisymmetry, holds by the construction. All four claims (\ref{pogeneration})--(\ref{poparent}) follow directly from the construction.
\end{proof}

The claim (\ref{pogeneration}) in Lemma \ref{potree} means that the partial order $\preceq$ is naturally formed by generations. The claim (\ref{poanchestor}) tells that each index pair $(k,\alpha)$ has a unique ancestor in the generation $l$. Together these claims imply that the partial order forms a tree structure. The claim (\ref{pochildren}) can be interpreted such that points corresponding to a parent and its child in the tree structure are close to each others, and the claim (\ref{poparent}) such that a parent is close to no other children.

In order to define the dyadic cubes, let us fix a partial order $\preceq$ satisfying the conditions in Lemma \ref{potree}. The dyadic cubes consist of balls with those reference points as centers whose index pair is a descendant of the index pair of the cube in the tree structure.  
\begin{definition}[Dyadic cube] \label{Qka}
The dyadic cube of generation $k \in \Z$ with index $\alpha \in I_k$ is
\begin{equation*}
\Qka := \bigcup_{(l,\beta) \preceq (k,\alpha)} B(z^l_\beta, a_0 \delta^l),
\end{equation*}
where $a_0 \in (0,1)$ is a parameter to be determined later.
\end{definition}

Dyadic cubes form the dyadic family of each generation $k$
\begin{equation*}
\D_k := \{\Qka : \alpha \in I_k\}
\end{equation*}
and the family of all dyadic cubes
\begin{equation*}
\D := \bigcup_{k \in \Z} \D_k.
\end{equation*}
The dyadic cubes in Definition \ref{Qka} satisfy many same kinds of properties as the classic dyadic cubes of $\R^n$ when the parameters $\delta$ and $a_0$ in their definitions are chosen to be small enough.
\begin{theorem}[Properties of dyadic cubes] \label{Qexist}
For the family $\D$, there exist constants $\delta \in (0,1)$, $a_0 \in (0,1)$, $\eta > 0$, $C_1 < \infty$, $C_2 < \infty$ and $N_0 \in \N$ depending only on $A_0$ and $A_1$ such that the following claims hold.
\begin{enumerate}[(a)]
\item \label{Qopen} Each $\Qka \in \D$ is open.
\item \label{Qball} Each $\Qka \in \D$ contains a ball $B(z^k_\alpha, a_0 \delta^k)$.
\item \label{Qdiam} For each $\Qka \in \D$ it holds $\diam(Q^k_\alpha) \leq C_1 \delta^k$.
\item \label{Qanchestor} For each $\Qka \in \D$ and $l < k$ there exists a unique $\beta \in I_l$ such that $\Qka \subset \Qlb$.
\item \label{Qtree} If $l \geq k$ and $\alpha \in I_k, \beta \in I_l$, then either $\Qlb \subset \Qka$ or $\Qlb \cap \Qka = \varnothing$.
\item \label{Qchildren} For each $\Qka \in \D$ it holds
\begin{equation*}
\# \{ Q^{k+1}_\beta \in \D_{k+1} : Q^{k+1}_\beta \subset \Qka \} \leq N_0.
\end{equation*}
\item \label{Qfull} For each $k \in \Z$ it holds
\begin{equation*}
\mu(X \setminus \bigcup_{\alpha \in I_k} \Qka) = 0.
\end{equation*}
\item \label{Qboundary} For each $\Qka \in \D$ it holds
\begin{equation*}
\mu(\{x \in \Qka : \rho(x, X \setminus \cl{\Qka}) \leq t \delta^k\}) \leq C_2 t^\eta \mu(\Qka) \quad \text{for each } t > 0,
\end{equation*}
where $\cl{\Qka}$ is the closure of $\Qka$.
\end{enumerate} 
\end{theorem}

The claims (\ref{Qball}) and (\ref{Qdiam}) in Theorem \ref{Qexist} tell that a dyadic cube contains a ball and, on the other hand, is contained in a ball whose radius is exponentially proportional to the index determining the generation of the cube. In other words, a dyadic cube is bounded by balls of the same magnitude both from inside and outside. Together with the property (\ref{doublingfinite}) of a doubling measure $\mu$ it implies that dyadic cubes have strictly positive and finite measure. The claims (\ref{Qanchestor}) and (\ref{Qtree}) mean that dyadic cubes form a natural tree structure determined by generations. The claim (\ref{Qchildren}) tells that the number of children of dyadic cubes has a common upper bound in the family tree. The claim (\ref{Qfull}), instead, means that each generation of dyadic cubes covers the whole space excluding a set of measure zero, and the claim (\ref{Qboundary}) that the measure of a dyadic cube is not accumulated close to its boundary. The last claim can be used mainly only when dealing with singular integral operators.

The values of parameters $\delta$ and $a_0$ affect the proof of Theorem \ref{Qexist} from its start to finish. Thus, they will not be fixed until the last claim of the theorem has been proved. However, the values are limited by constraints $\delta \in (0,\delta')$ and $a_0 \in (0, a_0')$ during the proof, where $\delta'$ and $a_0'$ depend only on $A_0$ and $A_1$. There will be finitely many such constraints, and thus the final upper bounds will be the minimums of the corresponding single ones. In most lemmas used in supporting the proof, we have to interpret that the lemmas hold providing that $\delta$ and $a_0$ are small enough even if it is not explicitly mentioned in the formulation of the lemmas.

We move to prove the claims in Theorem \ref{Qexist} one by one. The claim (\ref{Qopen}) follows directly from Definition \ref{Qka}: balls $B(\zlb,a_0 \delta^l)$ are open in the space of homogeneous type, and thus $\Qka$ is open as a union of open sets. Similarly, the claim (\ref{Qball}) follows directly from the definition because of the reflexivity property $(k,\alpha) \preceq (k,\alpha)$ of the partial order.

For the next claim, we need a weakened generalization of the property (\ref{potree}\ref{pochildren}) of the partial order $\preceq$, which gives an upper bound how far the center of a dyadic cube can be from the center of its ancestor.
\begin{lemma} \label{pochildrengeneral}
If $(l,\beta) \preceq (k,\alpha)$, then $\rho(\zlb,\zka) \leq 2 A_0 \delta^k$.
\end{lemma}
\begin{proof}
Assume $(l,\beta) \preceq (k,\alpha)$. Then there exists a unique chain
\begin{equation*}
(k,\alpha) = (k,\gamma_0) \succeq (k+1,\gamma_1) \succeq (k+2,\gamma_2) \succeq \dots \succeq (k+n,\gamma_n) = (l,\beta)
\end{equation*}
by Lemma \ref{potree}. By estimating with the triangle inequality (\ref{quasitriangle}) and the implication (\ref{potree}\ref{pochildren}), we get
\begin{eqnarray*}
\rho(\zka,\zlb) & \leq & A_0 \rho(\zka,z^{k+1}_{\gamma_1}) + A_0 \rho(z^{k+1}_{\gamma_1}, \zlb) \\
& \leq & A_0 \delta^k + A_0 \rho(z^{k+1}_{\gamma_1}, \zlb) \\
& \leq & A_0 \delta^k + A_0^2 \rho(z^{k+1}_{\gamma_1},z^{k+2}_{\gamma_2}) + A_0^2 \rho(z^{k+2}_{\gamma_2},\zlb) \\
& \leq & A_0 \delta^k + A_0^2 \delta^{k+1} + A_0^2 \rho(z^{k+2}_{\gamma_2},\zlb) \\
& \vdots & \\
& \leq & A_0 \delta^k + A_0^2 \delta^{k+1} + A_0^3 \delta^{k+2} + \dots + A_0^{n-1} \delta^{k+n-2} + A_0^{n-1} \delta^{k+n-1} \\
& \leq & A_0 \delta^k \sum_{j=0}^\infty (A_0 \delta)^j = \frac{A_0 \delta^k}{1 - A_0 \delta} \\
& \leq & 2 A_0 \delta^k,
\end{eqnarray*}
where $\delta$ has been chosen to be smaller than $\frac{1}{2 A_0}$. This choice also ensures that the geometric series in the second line from the bottom converges.
\end{proof}

Let us prove the claim (\ref{Qdiam}) in Theorem \ref{Qexist}. Let $x,y \in \Qka$. Then $x \in B(\zlb, a_0 \delta^l)$ and $y \in B(z^m_\gamma, a_0 \delta^m)$ for some $(l,\beta), (m,\gamma) \preceq (k,\alpha)$ by Definition \ref{Qka}. Hence
\begin{eqnarray*}
\rho(x,y) & \leq & A_0 \rho(x, z^l_\beta) + A_0^2 \rho(z^l_\beta, z^k_\alpha) + A_0^3 \rho(z^k_\alpha, z^m_\gamma) + A_0^3 \rho(z^m_\gamma, y) \\
& \leq & A_0 a_0 \delta^l + A_0^2 2 A_0 \delta^k + A_0^3 2 A_0 \delta^k + A_0^3 a_0 \delta^m \\
& \leq & A_0 1 \delta^k + A_0^2 2 A_0 \delta^k + A_0^3 2 A_0 \delta^k + A_0^3 1 \delta^k \\
& = & (A_0 + 3 A_0^3 + 2 A_0^4) \delta^k \\
& = & C_1 \delta^k.
\end{eqnarray*}
In the first inequality, we have used the triangle inequality (\ref{quasitriangle}) three times. In the second one, we have used Lemma \ref{pochildrengeneral} and the information that $x \in B(\zlb, a_0 \delta^l)$ and $y \in B(z^m_\gamma, a_0 \delta^m)$. In the last inequality, in turn, $a_0 \leq 1$ and $l,m \geq k$. By taking the supremum over the set $\{(x,y) : x,y \in \Qka\}$ of both sides of the derived inequality $\rho(x,y) \leq C_1 \delta^k$, we get the desired inequality $\diam(Q^k_\alpha) \leq C_1 \delta^k$. The claim (\ref{Qexist}\ref{Qdiam}) directly implies that $\Qka \subset B(\zka, C_1 \delta^k)$.

Before proving the next claim, we need a lemma telling that free parameters can be restricted such that dyadic cubes of the same generation become disjoint.
\begin{lemma} \label{QkaQkb}
If $\Qka \cap Q^k_\beta \neq \varnothing$ then $\alpha = \beta$.
\end{lemma}
\begin{proof}
Let $x \in \Qka \cap Q^k_\beta$. Then there exist pairs $(m,\gamma) \preceq (k,\alpha)$ and $(n,\sigma) \preceq (k,\beta)$ such that $x \in B(z^m_\gamma, a_0 \delta^m)$ and $x \in B(z^n_\sigma, a_0 \delta^n)$ by Definition \ref{Qka}. Without loss of generality, we may assume $m \geq n$. Then by the triangle inequality (\ref{quasitriangle}), we get
\begin{eqnarray}
\nonumber \rho(z^m_\gamma, z^n_\sigma) & \leq & A_0 \rho(z^m_\gamma, x) + A_0 \rho(x, z^n_\sigma) \\
\label{zmgzns} & \leq & A_0 a_0 \delta^m + A_0 a_0 \delta^n \\
\nonumber & \leq & 2 A_0 a_0 \delta^n.
\end{eqnarray}

Let us consider two cases. If $m = n$ the previous inequality (\ref{zmgzns}) becomes $\rho(z^n_\gamma, z^n_\sigma) < \delta^n$ when we have chosen $a_0 < \frac{1}{2 A_0}$. This implies that $\gamma = \sigma$, i.e. $(m,\gamma) = (n,\sigma)$, by the choice of $Z_n$ (\ref{zka}). Now the pairs $(m,\gamma)$ and $(n,\sigma)$ have the same unique ancestor in the generation $k$ by the property (\ref{potree}\ref{poanchestor}) of the partial order $\preceq$. Thus $\alpha = \beta$.

On the other hand, if $m > n$ there exists a unique $z^{n+1}_\tau$ such that $(m,\gamma) \preceq (n+1,\tau)$ by the property (\ref{potree}\ref{poanchestor}) of the partial order $\preceq$. By the triangle inequality (\ref{quasitriangle}), Lemma \ref{pochildrengeneral} and the inequality (\ref{zmgzns}), we get
\begin{eqnarray*}
\rho(z^{n+1}_\tau, z^n_\sigma) & \leq & A_0 \rho(z^{n+1}_\tau, z^m_\gamma) + A_0 \rho(z^m_\gamma, z^n_\sigma) \\
& \leq & A_0 2 A_0 \delta^{n+1} + A_0 2 A_0 a_0 \delta^n \\
& = & 2 A_0^2 (\delta + a_0) \delta^n \\
& < & \frac{1}{2 A_0} \delta^n,
\end{eqnarray*}
where we have chosen $\delta$ and $a_0$ to be smaller than $\frac{1}{8 A_0^3}$. Now, the property (\ref{potree}\ref{poparent}) of the partial order $\preceq$ implies $(n+1,\tau) \preceq (n,\sigma)$, and thus we get a chain
\begin{equation*}
(m,\gamma) \preceq (n+1,\tau) \preceq (n,\sigma) \preceq (k,\beta).
\end{equation*}
Because it also holds $(m,\gamma) \preceq (k,\alpha)$, we again deduce that $\alpha = \beta$ by the property (\ref{potree}\ref{poanchestor}).
\end{proof}

Next we prove the claim (\ref{Qtree}) in Theorem \ref{Qexist}. Let $l \geq k$ and $\Qlb \cap \Qka \neq \varnothing$. Choose $\gamma$ by the property (\ref{potree}\ref{poanchestor}) of the partial order $\preceq$ such that $(l,\beta) \preceq (k,\gamma)$. Then $\Qlb \subset Q^k_\gamma$ because of Definition \ref{Qka} and transitivity of $\preceq$. Thus, it also holds $Q^k_\gamma \cap \Qka \neq \varnothing$ which implies $\gamma = \alpha$ by Lemma \ref{QkaQkb}. Hence $\Qlb \subset \Qka$. On the other hand, if $l \geq k$ and $\Qlb \cap \Qka = \varnothing$, the case $\Qlb \subset \Qka$ is impossible since $\Qlb \neq \varnothing$. Thus, the claim (\ref{Qexist}\ref{Qtree}) has been proved.

By combining Lemma \ref{QkaQkb} and the property (\ref{Qexist}\ref{Qtree}) of dyadic cubes just proved, we get a clear connection between the partial order $\preceq$ and the dyadic family $\D$.
\begin{lemma} \label{connection}
Suppose $l \geq k$ and $\Qka, \Qlb \in \D$. Then
\begin{eqnarray}
\label{connection1} (l,\beta) \preceq (k,\alpha) & \text{if and only if} & \Qlb \subset \Qka, \quad \text{and} \\
\label{connection2} (l,\beta) \npreceq (k,\alpha) & \text{if and only if} & \Qlb \cap \Qka = \varnothing.
\end{eqnarray}
\end{lemma}
\begin{proof}
We prove the claim (\ref{connection1}) first. Suppose that $\Qlb \subset \Qka$. By the property (\ref{potree}\ref{poanchestor}) of the partial order $\preceq$, there exists a unique $\gamma \in I_k$ such that $(l,\beta) \preceq (k,\gamma)$. Then Definition \ref{Qka} and transitivity of $\preceq$ imply that $\Qlb \subset Q^k_\gamma$. Because now $\Qlb \subset \Qka \cap Q^k_\gamma$ and $\Qlb \neq \varnothing$, it holds $\Qka \cap Q^k_\gamma \neq \varnothing$, and thus $\alpha = \gamma$ by Lemma \ref{QkaQkb}. Hence $(l,\beta) \preceq (k,\alpha)$. The other direction follows directly from Definition \ref{Qka} and transitivity of the partial order $\preceq$.

The latter claim (\ref{connection2}) follows from the former one in a straightforward way; By negating (\ref{connection1}) we get that $(l,\beta) \npreceq (k,\alpha)$ if and only if $\Qlb \not \subset \Qka$. The latter proposition is true if and only if $\Qlb \cap \Qka = \varnothing$ by the property (\ref{Qexist}\ref{Qtree}) of dyadic cubes, and thus the claim has been proved.
\end{proof}

Let us prove the claim (\ref{Qanchestor}) in Theorem \ref{Qexist}. Let $\Qka \in \D$ and $l < k$. The property (\ref{potree}\ref{poanchestor}) of the partial order $\preceq$ gives a unique index $\beta \in I_l$ such that $(k,\alpha) \preceq (l,\beta)$. Then $\Qka \subset \Qlb$ holds exactly with that $\beta$ by the part (\ref{connection1}) of Lemma \ref{connection}, which proves the claim (\ref{Qexist}\ref{Qanchestor}).

Next we prove the claim (\ref{Qchildren}) in Theorem \ref{Qexist}. Let us consider the dyadic cube $\Qka \in \D$. For the center set $Z_{k+1} \subset X$, it holds
\begin{equation*}
\rho(z^{k+1}_\beta, z^{k+1}_\gamma) \geq \delta^{k+1} \quad \text{when } z^{k+1}_\beta, z^{k+1}_\gamma \in Z_{k+1}, z^{k+1}_\beta \neq z^{k+1}_\gamma
\end{equation*}
by its property (\ref{zka}). Thus, there exists a constant $N_0 \in \N$ such that $Z_{k+1}$ contains at most $N_0$ points in the ball $B(z^k_\alpha, (C_1 \delta^{-1}) \delta^{k+1}) = B(z^k_\alpha, C_1 \delta^k)$ by Lemma \ref{assouad}. Here $N_0$ depends only on the constants $A_0$ and $A_1$, as long as $\delta$ depends only on them. Because $\Qka \subset B(z^k_\alpha, C_1 \delta^k)$ by the property (\ref{Qexist}\ref{Qdiam}) of dyadic cubes, $Z_{k+1}$ also contains at most $N_0$ points in the cube $\Qka$. Hence, by noticing the property (\ref{Qexist}\ref{Qtree}) of dyadic cubes, it holds
\begin{equation*}
\# \{ Q^{k+1}_\beta \in \D_{k+1} : Q^{k+1}_\beta \subset \Qka \} = \# \{ z^{k+1}_\beta \in Z_{k+1} : z^{k+1}_\beta \in \Qka \} \leq N_0
\end{equation*}
and the claim (\ref{Qexist}\ref{Qanchestor}) has been proved.

Next we move to prove the claim (\ref{Qfull}) in Theorem \ref{Qexist}. Fix $k \in \Z$ and denote
\begin{equation*}
G := \bigcup_{\alpha \in I_k} \Qka.
\end{equation*}
For any $x \in X$ and $n \in \Z$, there exists $z^n_\beta \in Z_n$ such that $\rho(x,z^n_\beta) < \delta^n$ by (\ref{zkamaximality}). When $n \geq k$, there exists $\alpha \in I_k$ such that $(n,\beta) \preceq (k,\alpha)$ by the property (\ref{Qexist}\ref{Qanchestor}) of dyadic cubes, and thus $B(z^n_\beta, a_0 \delta^n) \subset G$ by Definition \ref{Qka}. Let us show that it also holds $B(z^n_\beta, a_0 \delta^n) \subset B(x, A_0 (1+a_0) \delta^n)$: let $y \in B(z^n_\beta, a_0 \delta^n)$, which together with the triangle inequality (\ref{quasitriangle}) implies
\begin{eqnarray*}
\rho(x,y) & \leq & A_0 \rho(x,z^n_\beta) + A_0 \rho(z^n_\beta, y) \\
& < & A_0 \delta^n + A_0 a_0 \delta^n \\
& = & A_0 (1+a_0) \delta^n.
\end{eqnarray*}
Hence $y \in B(x, A_0 (1+a_0) \delta^n)$.

Next we show that also the opposite relation holds for the measure $\mu$, i.e.
\begin{equation*}
\mu(B(z^n_\beta, a_0 \delta^n)) \geq c \mu \big( B(x, A_0 (1+a_0) \delta^n) \big),
\end{equation*}
where $c > 0$ is a constant depending only on $A_0$ and $A_1$. By using the doubling property (\ref{doubling}) $d$ times, we get
\begin{equation*}
\mu(B(z^n_\beta, 2^d a_0 \delta^n)) \leq A_1^d \mu(B(z^n_\beta, a_0 \delta^n)).
\end{equation*}
Choose $d$ such that $2^d a_0 \geq A_0^2(1 + a_0) + A_0$, and let $y \in B(x, A_0(1+a_0) \delta^n)$. Now, by applying the choice of $z^n_\beta$ and the triangle inequality (\ref{quasitriangle}), we get
\begin{eqnarray*}
\rho(y, z^n_\beta) & \leq & A_0 \rho(y,x) + A_0 \rho(x,z^n_\beta) \\
& < & A_0 A_0(1 + a_0) \delta^n + A_0 \delta^n \\
& = & (A_0^2(1+a_0) + A_0) \delta^n \\
& \leq & 2^d a_0 \delta^n,
\end{eqnarray*}
i.e. $y \in B(z^n_\beta, 2^d a_0 \delta^n)$. Hence $B(x, A_0(1+a_0) \delta^n) \subset B(z^n_\beta, 2^d a_0 \delta^n)$. Thus, by monotonicity and the doubling property of $\mu$, we get
\begin{eqnarray*}
\mu(B(z^n_\beta, a_0 \delta^n)) & \geq & \frac{1}{A_1^d} \mu(B(z^n_\beta, 2^d a_0 \delta^n)) \\
& \geq & \frac{1}{A_1^d} \mu \big( B(x, A_0(1+a_0) \delta^n) \big) \\
& = & c \mu \big( B(x, A_0(1+a_0) \delta^n) \big),
\end{eqnarray*}
where the constant $c$ depends only on $A_0$ and $A_1$, as long as the parameter $a_0$ also depends only on them.

For the rest of the proof, denote $A_0 (1+a_0) \delta^n =: r_n$. By combining the inclusions $B(z^n_\beta, a_0 \delta^n) \subset G$ and $B(z^n_\beta, a_0 \delta^n) \subset B(x, r_n)$ and monotonicity of $\mu$, we get
\begin{equation*}
\mu(B(z^n_\beta, a_0 \delta^n)) \leq \mu(G \cap B(x, r_n)).
\end{equation*}
By including also the inequality $c \mu(B(x, r_n)) \leq \mu(B(z^n_\beta, a_0 \delta^n))$ and noticing that the only thing we assumed about $n$ was $n \geq k$, it follows
\begin{equation*}
\frac{\mu(G \cap B(x, r_n))}{\mu(B(x, r_n))} \geq c > 0 \quad \text{for each } n \geq k.
\end{equation*}
By taking the limes inferior of both sides as $n \to \infty$ and noticing that $x \in X$ was arbitrary, we get
\begin{equation} \label{Epositive}
\liminf_{r \to 0} \frac{\mu(G \cap B(x,r))}{\mu(B(x,r))} \geq c > 0 \quad \text{for each } x \in X.
\end{equation}
On the other hand, by choosing $f$ to be the charachteristic function $\chi_G$ in Lebesgue's theorem on differentiation \ref{lebesgue}, we get the equation
\begin{equation} \label{Echaracteristic}
\lim_{r \to 0} \frac{\mu(G \cap B(x,r))}{\mu(B(x,r))} = \chi_G(x) \quad \text{for almost every } x \in X.
\end{equation}
Combining the limit results (\ref{Epositive}) and (\ref{Echaracteristic}) implies that for almost every $x \in X$ it holds $\chi_G(x) = 1$. Thus $\mu(X \setminus G) = 0$. Hence
\begin{equation*}
\mu(X \setminus \bigcup_{\alpha \in I_k} \Qka) = 0
\end{equation*}
for any $k \in \Z$ and the claim (\ref{Qexist}\ref{Qfull}) has been proved.

By the property (\ref{Qexist}\ref{Qfull}) of dyadic cubes just proved, the sets
\begin{equation*}
N_k := X \setminus \bigcup_{\alpha \in I_k} \Qka
\end{equation*}
are null sets, i.e. for every generation $k$ the family of cubes $\D_k$ covers the space $X$ excluding a set of measure zero. Also $\mu(\bigcup_{k \in \Z} N_k) = 0$ by subadditivity of $\mu$. From now on, when we want to highlight that this null set has been excluded, we denote
\begin{equation*}
\hat{X} := X \setminus \bigcup_{k \in \Z} N_k.
\end{equation*}
Especially, it holds that $\hat{X} \subset \bigcup_{\alpha \in I_k} \Qka$ for each $k \in \Z$ and $\mu(X \setminus \hat{X}) = 0$.

Before proving the claim (\ref{Qboundary}) in Theorem \ref{Qexist}, we need several lemmas first. The first one is a stronger version of the property (\ref{Qexist}\ref{Qball}) of dyadic cubes telling that dyadic cubes also contain a bigger ball, excluding a set of measure zero, though.
\begin{lemma} \label{Qballstronger}
Denote $C_3 = \frac{1}{4 A_0^2}$. Then for each $\Qka \in \D$ it holds
\begin{equation*}
B(\zka, C_3 \delta^k) \cap \hat{X} \subset \Qka.
\end{equation*}
\end{lemma}
\begin{proof}
Let $x \in B(\zka, C_3 \delta^k) \cap \hat{X}$. Suppose $x \not \in \Qka$, which implies $x \in Q^k_\beta$ for some other $\beta \in I_k$ because $x \in \hat{X}$. Then there exists $(l,\gamma) \preceq (k,\beta)$ such that $x \in B(z^l_\gamma, a_0 \delta^l)$ by Definition \ref{Qka}. The tree structure (\ref{potree}\ref{pogeneration}--\ref{poanchestor}) of the partial order $\preceq$ implies that $l \geq k$ and $(l,\gamma) \npreceq (k,\alpha)$. Let us consider two different cases separately. If $l = k$, using the triangle inequality (\ref{quasitriangle}) and the location of $x$ implies
\begin{eqnarray*}
\rho(\zka, z^l_\gamma) & \leq & A_0 \rho(\zka, x) + A_0 \rho(x, z^l_\gamma) \\
& \leq & A_0 C_3 \delta^k + A_0 a_0 \delta^l \\
& = &  A_0 \frac{1}{4 A_0^2} \delta^k + A_0 a_0 \delta^k \\
& \leq & (\tfrac{1}{4} + A_0 a_0) \delta^k \\
& < & \delta^k,
\end{eqnarray*}
where $a_0$ has been chosen to be smaller than $\frac{3}{4 A_0}$. Because $l = k$, the derived inequality contradicts (\ref{zka}).

On the other hand, if $l > k$, there exists $\sigma \in I_{k+1}$ such that $(l,\gamma) \preceq (k+1,\sigma)$ by the property (\ref{potree}\ref{poanchestor}) of the partial order $\preceq$. Since $x \in B(z^l_\gamma, a_0 \delta^l)$, $x \in Q^{k+1}_\sigma$ by Definition \ref{Qka}. Because also $z^{k+1}_\sigma \in Q^{k+1}_\sigma$, it holds $\rho(x,z^{k+1}_\sigma) \leq C_1 \delta^{k+1}$ by the property (\ref{Qexist}\ref{Qdiam}) of dyadic cubes. Then the triangle inequality (\ref{quasitriangle}) and the original assumption $x \in B(\zka, C_3 \delta^k)$ imply
\begin{eqnarray*}
\rho(\zka, z^{k+1}_\sigma) & \leq & A_0 \rho(\zka, x) + A_0 \rho(x, z^{k+1}_\sigma) \\
& \leq & A_0 C_3 \delta^k + A_0 C_1 \delta^{k+1} \\
& = & A_0 \frac{1}{4 A_0^2} \delta^k + A_0 C_1 \delta^{k+1} \\
& = & \big( \frac{1}{4 A_0} + A_0 C_1 \delta \big) \delta^k \\
& < & \frac{1}{2 A_0} \delta^k,
\end{eqnarray*}
where $\delta$ has been chosen to be smaller than $\frac{1}{4 A_0^2 C_1}$. Now, $(k+1,\sigma) \preceq (k,\alpha)$ by the property (\ref{potree}\ref{poparent}) of the partial order $\preceq$, and thus $(l,\gamma) \preceq (k,\alpha)$ by transitivity. This is a contradiction since we have also deduced $(l,\gamma) \npreceq (k,\alpha)$.
\end{proof}

Next we move to consider the set in the claim (\ref{Qboundary}) of Theorem \ref{Qexist} containing the points close to the boundary of a dyadic cube. We denote such set by
\begin{equation} \label{Ekat}
E^k_\alpha(\tau) := \{x \in \Qka : \rho(x, X \setminus \cl{\Qka}) < \tau \delta^k\}
\end{equation}
and call it the $\tau$-boundary of the cube $\Qka \in \D$. First, we show that the set $X \setminus \cl{\Qka}$ in the definition of $\tau$-boundary can be controlled through the set $\hat{X} \setminus \Qka$, which is easier to deal with.
\begin{lemma} \label{XmQcontrol}
Suppose $\Qka \in \D$ and $x \in X$. Then
\begin{equation*}
\rho(x, \hat{X} \setminus \Qka) \leq A_0 \rho(x, X \setminus \cl{\Qka}).
\end{equation*}
\end{lemma}
\begin{proof}
If $X \setminus \cl{\Qka} = \varnothing$, then $\rho(x, X \setminus \cl{\Qka}) = \infty$ and the claim holds. So, let $X \setminus \cl{\Qka} \neq \varnothing$ and $y \in X \setminus \cl{\Qka}$. We show first that $\rho(y, \hat{X} \setminus \Qka) = 0$. Suppose that $\rho(y, \hat{X} \setminus \Qka) = r > 0$, and let $z \in B(y, \eps_1)$, where $\eps_1 > 0$ is a parameter to be determined later. Then by the triangle inequality (\ref{quasitriangle}), we get
\begin{eqnarray*}
r = \rho(y, \hat{X} \setminus \Qka) & \leq & A_0 \rho(y,z) + A_0 \rho(z, \hat{X} \setminus \Qka) \\
& \leq & A_0 \eps_1 + A_0 \rho(z, \hat{X} \setminus \Qka),
\end{eqnarray*}
from which we can solve
\begin{equation*}
\rho(z, \hat{X} \setminus \Qka) \geq \frac{r}{A_0} - \eps_1 = \frac{r}{2 A_0} > 0,
\end{equation*}
when we have chosen $\eps_1 = \frac{r}{2 A_0}$. Thus
\begin{equation*}
\rho(B(y, \eps_1), \hat{X} \setminus \Qka) > 0
\end{equation*}
for the whole ball $B(y, \eps_1)$.

On the other hand, because $y \in X \setminus \cl{\Qka}$ and $X \setminus \cl{\Qka}$ is an open set as the complement of a closed set, it contains some ball $B(y, \eps_2)$. Combining the derived deductions implies that there exists a ball $B = B(y, \min \{ \eps_1, \eps_2 \})$ such that 
\begin{equation*}
B \subset X \setminus \cl{\Qka} \quad \text{and} \quad \rho(B, \hat{X} \setminus \Qka) > 0.
\end{equation*}
Further, it implies
\begin{equation*}
B \subset (X \setminus \cl{\Qka}) \setminus (\hat{X} \setminus \Qka) \subset X \setminus \hat{X}.
\end{equation*}
Thus, $0 < \mu(B) \leq \mu(X \setminus \hat{X}) = 0$ by the property (\ref{doublingfinite}) and monotonicity of $\mu$, which is a contradiction. Hence $\rho(y, \hat{X} \setminus \Qka) = 0$, and thus estimating with the triangle inequality (\ref{quasitriangle}) gives
\begin{equation*}
\rho(x, \hat{X} \setminus \Qka) \leq A_0 \rho(x, y) + A_0 \rho(y, \hat{X} \setminus \Qka) = A_0 \rho(x, y).
\end{equation*}
The claim follows by taking the infimum over the set $\{y \in X \setminus \cl{\Qka}\}$ of the both sides.
\end{proof}

We show next that descendants of a dyadic cube in an arbitrary generation cover its $\tau$-boundary for small enough $\tau$.
\begin{lemma} \label{17step1}
For each $N \in \N$ there exists $\tau' > 0$ such that if $\tau \in (0, \tau')$ and $x \in E^k_\alpha(\tau)$, then $x \in Q^{k+N}_\sigma$ for some $\sigma \in I_{k+N}$ with $(k+N,\sigma) \preceq (k,\alpha)$.
\end{lemma}
\begin{proof}
Fix $N \in \N$ and let $x \in E^k_\alpha(\tau)$, where $\tau > 0$. Then $x \in B(\zlb, a_0 \delta^l)$ for some $(l,\beta) \preceq (k,\alpha)$ by Definition \ref{Qka}. Because of Lemmas \ref{Qballstronger} and \ref{connection}, we get
\begin{equation*}
B(\zlb, C_3 \delta^l) \cap \hat{X} \subset \Qlb \subset \Qka,
\end{equation*}
where $C_3 = \frac{1}{4 A_0^2}$. This implies
\begin{equation*}
\hat{X} \setminus \Qka \subset \hat{X} \setminus (B(\zlb, C_3 \delta^l) \cap \hat{X}) = \hat{X} \setminus B(\zlb, C_3 \delta^l),
\end{equation*}
and thus
\begin{eqnarray*}
\rho(\zlb, \hat{X} \setminus \Qka) & = & \inf_{y \in \hat{X} \setminus \Qka} \rho(\zlb, y) \\
& \geq & \inf_{y \in \hat{X} \setminus B(\zlb, C_3 \delta^l)} \rho(\zlb, y) \\
& \geq & C_3 \delta^l.
\end{eqnarray*}
On the other hand, by Lemma \ref{XmQcontrol}, the triangle inequality (\ref{quasitriangle}) and the choice of $x$
\begin{eqnarray*}
\rho(\zlb, \hat{X} \setminus \Qka) & \leq & A_0 \rho(\zlb, X \setminus \cl{\Qka}) \\
& \leq & A_0^2 \rho(\zlb, x) + A_0^2 \rho(x, X \setminus \cl{\Qka}) \\
& \leq & A_0^2 a_0 \delta^l + A_0^2 \tau \delta^k.
\end{eqnarray*}
By combining these two inequalities, we get $(C_3 - A_0^2 a_0) \delta^l \leq A_0^2 \tau \delta^k$. Choosing $a_0$ to be smaller than $\frac{1}{8 A_0^4}$ implies $C_3 - A_0^2 a_0 \geq \frac{1}{8 A_0^2} > 0$, and thus $\delta^l \leq 8 A_0^4 \tau \delta^k$. Now, when $\tau$ is chosen to be smaller than $\frac{1}{8 A_0^4} \delta^N =: \tau'$, we must have $l \geq k+N$. Then choose $\sigma \in I_{k+N}$ such that $(l,\beta) \preceq (k+N,\sigma)$. Because $x \in B(\zlb, a_0 \delta^l)$, also $x \in Q^{k+N}_\sigma$ by Definition \ref{Qka}. In addition, because $x \in \Qka \cap Q^{k+N}_\sigma$, it holds $Q^{k+N}_\sigma \subset \Qka$ by the property (\ref{Qexist}\ref{Qtree}) of dyadic cubes, and thus $(k+N,\sigma) \preceq (k,\alpha)$ by Lemma \ref{connection}. Hence the claim has been proved.
\end{proof}

By Lemma \ref{17step1} and the property (\ref{Qexist}\ref{Qanchestor}) of dyadic cubes, when $\tau > 0$ is small enough, for each point $x$ in $E^k_\alpha(\tau)$ there exists a unique chain from the generation $k$ to the generation $k+N$ such $x$ belongs to every dyadic cube in the chain. Denote the corresponding chain of index pairs by
\begin{equation} \label{CkNx}
\C_k^N(x) := \{(j,\sigma(x,j))\}_{j=k}^{k+N}, \quad x \in E^k_\alpha(\tau),
\end{equation}
where $\sigma(x,j) \in I_j$ such that $\sigma(x,k) = \alpha$, $x \in Q^j_{\sigma(x,j)}$ and
\begin{equation*}
(j,\sigma(x,j)) \preceq (j-1,\sigma(x,j-1)), \quad j = k+1, \dots, k+N.
\end{equation*}
\begin{lemma} \label{17step2}
Suppose that $\Qka \in \D$ and $N \in \N$. If $\tau$ is small enough and independent of $\Qka$, then there exists $\eps_1 > 0$ depending only on the constants $A_0$ and $A_1$ such that for each $x \in E^k_\alpha(\tau)$
\begin{equation*}
\rho(z^j_{\sigma_j},z^i_{\sigma_i}) \geq \eps_1 \delta^j \quad \text{when } (j,\sigma_j), (i,\sigma_i) \in \C_k^N(x) \text{ and } j < i.
\end{equation*}
\end{lemma}
\begin{proof}
Let $\tau \in (0,\tau')$ to be determined later, where $\tau'$ is determined by $N$ as in Lemma \ref{17step1}. Suppose that the claim does not hold, i.e. for each $\eps_1 > 0$ we have
\begin{equation*}
\rho(z^j_{\sigma_j}, z^i_{\sigma_i}) < \eps_1 \delta^j
\end{equation*}
for some $x \in E^k_\alpha(\tau)$ and $(j,\sigma_j), (i,\sigma_i) \in \C_k^N(x)$, where $j < i$. For consistency, denote $\sigma_k := \sigma(x,k) = \alpha$. Then Lemma \ref{XmQcontrol}, the triangle inequality (\ref{quasitriangle}) and the property (\ref{Qexist}\ref{Qdiam}) of dyadic cubes imply
\begin{eqnarray*}
\rho(z^j_{\sigma_j}, \hat{X} \setminus Q^k_{\sigma_k}) & \leq & A_0 \rho(z^j_{\sigma_j}, X \setminus \cl{Q^k_{\sigma_k}})\\
& \leq & A_0^2 \rho(z^j_{\sigma_j}, x) + A_0^2 \rho(x, X \setminus \cl{Q^k_{\sigma_k}}) \\
& \leq & A_0^3 \rho(z^j_{\sigma_j}, z^i_{\sigma_i}) + A_0^3 \rho(z^i_{\sigma_i}, x) + A_0^2 \rho(x, X \setminus \cl{Q^k_{\sigma_k}}) \\
& < & A_0^3 \eps_1 \delta^j + A_0^3 C_1 \delta^i + A_0^2 \tau \delta^k \\
& = & (A_0^3 \eps_1 + A_0^3 C_1 \delta^{i-j} + A_0^2 \tau \delta^{k-j}) \delta^j \\
& \leq & (A_0^3 \eps_1 + A_0^3 C_1 \delta + A_0^2 \tau \delta^{-N}) \delta^j,
\end{eqnarray*}
which holds for each $\eps_1 > 0$ by our assumption and for each $\tau \in (0, \tau')$ by Lemma \ref{17step1}. Choose $\eps_1$ such that $A_0^3 \eps_1 < \frac{1}{3} C_3$, $\delta$ such that $A_0^3 C_1 \delta < \frac{1}{3} C_3$, and $\tau$ such that $A_0^2 \tau \delta^{-N} < \frac{1}{3} C_3$, where $C_3 = \frac{1}{4 A_0^2}$ as in Lemma \ref{Qballstronger}. Then
\begin{equation} \label{3C3}
\rho(z^j_{\sigma_j}, \hat{X} \setminus Q^k_{\sigma_k}) < (\frac{1}{3} C_3 + \frac{1}{3} C_3 + \frac{1}{3} C_3) \delta^j = C_3 \delta^j.
\end{equation}
Further, by Lemma \ref{Qballstronger} and the definition of the chain $\C_k^N(x)$, it holds
\begin{equation*}
B(z^j_{\sigma_j}, C_3 \delta^j) \cap \hat{X} \subset Q^j_{\sigma_j} \subset Q^k_{\sigma_k},
\end{equation*}
i.e. $\hat{X} \setminus Q^k_{\sigma_k} \subset \hat{X} \setminus B(z^j_{\sigma_j}, C_3 \delta^j)$. Combining this with the inequality (\ref{3C3}) implies
\begin{equation*}
\rho(z^j_{\sigma_j}, \hat{X} \setminus B(z^j_{\sigma_j}, C_3 \delta^j)) \leq \rho(z^j_{\sigma_j}, \hat{X} \setminus Q^k_{\sigma_k}) < C_3 \delta^j,
\end{equation*}
which is a contradiction since the distance between the center and the complement of the same ball can not be smaller than its radius. Thus, the claim has been proved.
\end{proof}

Denote the centers of dyadic cubes in the chain (\ref{CkNx}) which are close to the $\tau$-boundary (\ref{Ekat}) of the cube $\Qka$ by
\begin{equation} \label{Sjt}
S_j(\tau) := \bigcup_{x \in E^k_\alpha(\tau)} \{z^j_{\sigma_j} : (j,\sigma_j) \in \C_k^N(x)\}, \quad k \leq j \leq k+N.
\end{equation}
\begin{lemma} \label{17step3}
There exists $\tau > 0$, and $\eps_2 > 0$ depending only on $A_0$ and $A_1$ such that
\begin{equation*}
B(z, \eps_2 \delta^i) \cap B(z', \eps_2 \delta^j) = \varnothing \quad \text{for each } z \in S_i(\tau), z' \in S_j(\tau), z \neq z'.
\end{equation*}
\end{lemma}
\begin{proof}
Choose $\tau$ to be the small number determined by Lemma \ref{17step2}. Fix indices $i$ and $j$ such that $k \leq j \leq i \leq k + N$, and centers $z \in S_i(\tau), z' \in S_j(\tau)$. If $z$ and $z'$ belong to different chains $\C_k^N(\cdot)$, then $z = z^i_{\sigma(x,i)}, z' = z^j_{\sigma(x',j)}$ for some $x, x' \in E^k_\alpha(\tau)$, where $(i,\sigma(x,i)) \npreceq (j,\sigma(x',j))$. Then $Q^i_{\sigma(x,i)} \cap Q^j_{\sigma(x',j)} = \varnothing$ because of Lemma \ref{connection}. This implies 
\begin{equation*}
B(z^i_{\sigma(x,i)}, a_0 \delta^i) \cap B(z^j_{\sigma(x',j)}, a_0 \delta^j) = \varnothing
\end{equation*}
by the property (\ref{Qexist}\ref{Qball}) of dyadic cubes, and thus we can choose $\varepsilon_2 \leq a_0$. Hence the first case has been proved.

If $z$ and $z'$, instead, belong to the same chain $\C_k^N(\cdot)$, then there exists $x \in E^k_\alpha(\tau)$ such that $z = z^i_{\sigma(x,i)}, z' = z^j_{\sigma(x,j)}$. In this case, the strict inequality $j < i$ must hold due to the assumption $z \neq z'$. This implies $\rho(z,z') \geq \varepsilon_1 \delta^j$ for some $\eps_1 > 0$ by Lemma \ref{17step2}. On the other hand, if we assume that for each $\varepsilon_2 > 0$ there exists $y \in B(z, \varepsilon_2 \delta^i) \cap B(z', \varepsilon_2 \delta^j)$, estimating with the triangle inequality (\ref{quasitriangle}) gives
\begin{eqnarray*}
\rho(z,z') & \leq & A_0 \rho(z,y)+ A_0 \rho(y,z') \\
& < & A_0 \eps_2 \delta^i + A_0 \eps_2 \delta^j \\
& \leq & 2 A_0 \eps_2 \delta^j \\
& = & \eps_1 \delta^j,
\end{eqnarray*}
where we have chosen $\eps_2 = \frac{\eps_1}{2 A_0}$. In this case, we get a contradiction, and thus the claim has been proved.
\end{proof}

The ratio of the measures of a $\tau$-boundary and the corresponding cube becomes arbitrarily small when $\tau$ is chosen small enough.
\begin{lemma} \label{17step4}
For each $\eps > 0$ there exists $\tau > 0$ such that
\begin{equation*}
\mu(E^k_\alpha(\tau)) < \eps \mu(\Qka) \quad \text{for each } \Qka \in \D.
\end{equation*}
\end{lemma}
\begin{proof}
Fix $\eps > 0$ and $\Qka \in \D$. Let $N \in \N$ be a large number to be determined later. Let $\eps_2$ be small enough determined by Lemma \ref{17step3} and $\tau$ determined by $N$ such that such $\eps_2$ exists. First, we show that
\begin{equation*}
E^k_\alpha(\tau) \subset \bigcup_{z \in S_{k+N}(\tau)} B(z, C_1 \delta^{k+N}).
\end{equation*}
If $x \in E^k_\alpha(\tau)$, then $x \in Q^{k+N}_\sigma$ for some $(k+N,\sigma) \in \C_k^N(x)$ by Lemma \ref{17step1}. Then $\rho(x,z^{k+N}_\sigma) < C_1 \delta^{k+N}$ by the property (\ref{Qexist}\ref{Qdiam}) of dyadic cubes, i.e. $x \in B(z^{k+N}_\sigma, C_1 \delta^{k+N})$. Because $z^{k+N}_\sigma \in S_{k+N}(\tau)$, we further get
\begin{equation*}
x \in \bigcup_{z \in S_{k+N}(\tau)} B(z, C_1 \delta^{k+N}).
\end{equation*}
Then by monotonicity, subadditivity and the doubling property (\ref{doubling}) of $\mu$,
\begin{eqnarray}
\nonumber \mu(E^k_\alpha(\tau)) & \leq & \mu \big( \bigcup_{z \in S_{k+N}(\tau)} B(z, C_1 \delta^{k+N}) \big) \\
\label{17step4ie1} & \leq & \sum_{z \in S_{k+N}(\tau)} \mu(B(z, C_1 \delta^{k+N})) \\
\nonumber & \leq & C \sum_{z \in S_{k+N}(\tau)} \mu(B(z, \eps_2 \delta^{k+N})),
\end{eqnarray}
where $C = A_1^d$ when applied (\ref{doubling}) $d$ times by choosing $d \in \N$ such that $2^d \eps_2 \geq C_1$.

Next, let $k \leq j \leq k+N$ and denote $z \preceq w$ to mean that $(l,\beta) \preceq (m,\gamma)$ when $z = z^l_\beta$ and $w = z^m_\gamma$ are the centers of dyadic cubes. Because the partial order $\preceq$ forms a tree structure and the balls $B(z, \eps_2 \delta^{k+N})$ and $B(z', \eps_2 \delta^{k+N})$, $z \neq z'$, are disjoint by Lemma \ref{17step3}, we can split the sum in the above inequality as follows:
\begin{equation} \label{17step4ie2}
\sum_{z \in S_{k+N}(\tau)} \mu(B(z, \eps_2 \delta^{k+N})) = \sum_{w \in S_j(\tau)} \sum_{\substack{z \in S_{k+N}(\tau) \\ z \preceq w}} \mu(B(z, \eps_2 \delta^{k+N})).
\end{equation}
By the fact $\eps_2 \leq a_0$, Definition \ref{Qka} and the property (\ref{Qexist}\ref{Qdiam}) of dyadic cubes, the balls in the sum satisfy
\begin{equation*}
B(z, \eps_2 \delta^{k+N}) \subset B(z, a_0 \delta^{k+N}) \subset Q^j(w) \subset B(w, C_1 \delta^j),
\end{equation*}
where $Q^j(w) \in \D_j$ is the cube whose center is $w$. On the other hand, the balls $B(z, \eps_2 \delta^{k+N})$ are disjoint by Lemma \ref{17step3}, and thus
\begin{equation} \label{17step4ie3}
\sum_{\substack{z \in S_{k+N}(\tau) \\ z \preceq w}} \mu(B(z, \eps_2 \delta^{k+N})) \leq \mu(B(w, C_1 \delta^j))
\end{equation}
by additivity and monotonicity of the measure $\mu$.

By combining the inequalities (\ref{17step4ie1})--(\ref{17step4ie3}), applying the doubling condition again between $C_1$ and $\eps_2$ and noticing that the balls $B(w, \eps_2 \delta^j)$ are disjoint by Lemma \ref{17step3}, we get
\begin{eqnarray*}
\mu(E^k_\alpha(\tau)) & \leq & C \sum_{z \in S_{k+N}(\tau)} \mu(B(z, \eps_2 \delta^{k+N})) \\
& = & C \sum_{w \in S_j(\tau)} \sum_{\substack{z \in S_{k+N}(\tau) \\ z \preceq w}} \mu(B(z, \eps_2 \delta^{k+N})) \\
& \leq & C \sum_{w \in S_j(\tau)} \mu(B(w, C_1 \delta^j)) \\
& \leq & C^2 \sum_{w \in S_j(\tau)} \mu(B(w, \eps_2 \delta^j)) \\
& = & C^2 \mu \big( \bigcup_{w \in S_j(\tau)} B(w, \eps_2 \delta^j) \big).
\end{eqnarray*}

Denote $G_j := \bigcup_{z \in S_j(\tau)} B(z, \eps_2 \delta^j)$. Then the derived inequality becomes
\begin{equation*}
\mu(E^k_\alpha(\tau)) \leq C^2 \mu(G_j), \quad k \leq j \leq k+N.
\end{equation*}
The sets $G_j$ are subsets of the original cube $\Qka$ by Definition \ref{Qka} and the fact $\eps_2 \leq a_0$. In addition, they are disjoint by Lemma \ref{17step3}. Therefore
\begin{equation*}
\mu(\Qka) \geq \mu \big( \bigcup_{j=k}^{k+N} G_j \big) = \sum_{j=k}^{k+N} \mu(G_j) \geq \sum_{j=k}^{k+N} \frac{1}{C^2} \mu(E^k_\alpha(\tau)) \geq \frac{N}{C^2} \mu(E^k_\alpha(\tau))
\end{equation*}
by monotonicity and additivity of the measure $\mu$. Choosing $N$ to be greater than $\frac{C^2}{\eps}$ implies $\mu(E^k_\alpha(\tau)) < \eps \mu(\Qka)$, and thus the claim has been proved.
\end{proof}

Denote the descendants of each $\Qka \in \D$ in the generation $k+j$ that are close to its boundary as follows:
\begin{equation*}
\mathcal{E}_j(\Qka) := \{Q^{k+j}_\beta \subset \Qka : \rho(Q^{k+j}_\beta, X \setminus \cl{\Qka}) \leq C_4 \delta^{k+j}\},
\end{equation*}
where $C_4$ is a large constant to be determined later. Denote
\begin{equation*}
E_j(\Qka) := \{x : x \in Q^{k+j}_\beta \text{ for some } Q^{k+j}_\beta \in \mathcal{E}_j(\Qka)\}
\end{equation*}
for the corresponding set of points. The set $E_j(\Qka)$ closely corresponds to the $\tau$-boundary $E^k_\alpha(\tau)$ of the cube $\Qka$.
\begin{lemma} \label{17step5}
Denote $C_5 = \frac{1}{8 A_0^4}$. When $C_4$ is chosen large enough and $\tau$ and $j$ have the connection $C_5 \delta^{j+1} < \tau \leq C_5 \delta^j$, then
\begin{equation*}
E^k_\alpha(\tau) \subset E_j(\Qka) \subset E^k_\alpha(C_6 \tau) \quad \text{for each } \Qka \in \D,
\end{equation*}
where $C_6$ is a constant independent of $j$ and $\tau$.
\end{lemma}
\begin{proof}
Fix $\Qka \in \D$. We start by showing the first inclusion of the claim. Take $x \in E^k_\alpha(\tau)$. Then $x \in Q^{k+j}_\sigma$ for some $\sigma \in I_{k+j}$, especially when $\tau \leq \frac{1}{8 A_0^4} \delta^j = C_5 \delta^j$, by Lemma \ref{17step1} and its proof. Thus, $\rho(z^{k+j}_\sigma, x) \leq C_1 \delta^{k+j}$ by the property (\ref{Qexist}\ref{Qdiam}) of dyadic cubes, and by estimating with the triangle inequality (\ref{quasitriangle}), we get
\begin{eqnarray*}
\rho(Q^{k+j}_\sigma, X \setminus \cl{\Qka}) & \leq & A_0 \rho(Q^{k+j}_\sigma, z^{k+j}_\sigma) + A_0^2 \rho(z^{k+j}_\sigma, x) + A_0^2 \rho(x, X \setminus \cl{\Qka}) \\
& \leq & 0 + A_0^2 C_1 \delta^{k+j} + A_0^2 \tau \delta^k \\
& \leq & A_0^2 C_1 \delta^{k+j} + A_0^2 C_5 \delta^j \delta^k \\
& = & A_0^2 (C_1 + C_5) \delta^{k+j}.
\end{eqnarray*}
By choosing $C_4$ to be the coefficient $A_0^2 (C_1 + C_5)$ above, it holds $Q^{k+j}_\sigma \in \mathcal{E}_j(\Qka)$ implying also $x \in E_j(\Qka)$.

Next, we show the latter inclusion of the claim. Take $x \in E_j(\Qka)$. Then $x \in Q^{k+j}_\beta$ for some $Q^{k+j}_\beta \in \mathcal{E}_j(\Qka)$. Let $y \in Q^{k+j}_\beta$, and estimate with the triangle inequality (\ref{quasitriangle}) and the property (\ref{Qexist}\ref{Qdiam}) of dyadic cubes \begin{eqnarray*}
\rho(x, X \setminus \cl{\Qka}) & \leq & A_0 \rho(x,y) + A_0 \rho(y, X \setminus \cl{\Qka}) \\
& \leq & A_0 C_1 \delta^{k+j} + A_0 \rho(y, X \setminus \cl{\Qka}).
\end{eqnarray*}
By taking the infimum over the set $\{ y \in Q^{k+j}_\beta \}$ of the both sides of the derived inequality, we get
\begin{eqnarray*}
\rho(x, X \setminus \cl{\Qka}) & \leq & A_0 C_1 \delta^{k+j} + A_0 \rho(Q^{k+j}_\beta, X \setminus \cl{\Qka}) \\
& \leq & A_0 C_1 \delta^{k+j} + A_0 C_4 \delta^{k+j} \\
& = & A_0 (C_1 + C_4) \delta^{-1} \delta^{j+1} \delta^k \\
& \leq &  A_0 (C_1 + C_4) \delta^{-1} \frac{\tau}{C_5} \delta^k \\[-1ex]
& = & C_6 \tau \delta^k,
\end{eqnarray*}
where we know $Q^{k+j}_\beta \in \mathcal{E}_j(\Qka)$ and $C_5 \delta^{j+1} \leq \tau$. Because in addition $x \in \Qka$, we deduce $x \in E^k_\alpha(C_6 \tau)$.
\end{proof}

Let us finally prove the claim (\ref{Qboundary}) of Theorem \ref{Qexist}. To prove the claim for small values of $t$, i.e.
\begin{equation*}
\mu(E^k_\alpha(t)) \leq C_2 t^\eta \mu(\Qka) \quad \text{for each } 0 < t \leq C_5,
\end{equation*}
it suffices to show that there exist $C$ and $\eta$ such that
\begin{equation} \label{Qboundary2}
\mu(E_j(\Qka)) \leq C \delta^{j \eta} \mu(\Qka) \quad \text{for each } j \geq 0.
\end{equation}
Namely, if we assume that (\ref{Qboundary2}) holds, we can choose the relation
\begin{equation*}
C_5 \delta^{j+1} < t \leq C_5 \delta^j
\end{equation*}
between $j$ and $t$. Then the set $\{j \geq 0\}$ corresponds to the set $\{0 < t \leq C_5\}$, and by Lemma \ref{17step5} and monotonicity of the measure $\mu$, we get
\begin{eqnarray*}
\mu(E^k_\alpha(t)) & \leq & \mu(E_j(\Qka))\\
& \leq & C \delta^{j \eta} \mu(\Qka) \\
& = & C \delta^{-\eta} \delta^{(j+1) \eta} \mu(\Qka) \\
& \leq & C \delta^{-\eta} \big( \frac{t}{C_5} \big)^\eta \mu(\Qka) \\
& \leq & C_2 t^\eta \mu(\Qka),
\end{eqnarray*}
where $C_2$ has been chosen to be greater than or equal to $C \delta^{-\eta} C_5^{-\eta}$.

To prove the claim (\ref{Qboundary2}), fix a large index $J \in \N$ for which
\begin{equation} \label{largeJ}
\mu(E_J(\Qka)) \leq \frac{1}{2} \mu(\Qka) \quad \text{for each } \Qka \in \D.
\end{equation}
Such $J$ exists since there exists $\tau > 0$ such that
\begin{equation*}
\mu(E^k_\alpha(C_6 \tau)) \leq \frac{1}{2} \mu(\Qka) \quad \text{for each } \Qka \in \D
\end{equation*}
by Lemma \ref{17step4}, and further there exists an index $J$ determined by $\tau$ such that
\begin{equation*}
\mu(E_J(\Qka)) \leq \mu(E^k_\alpha(C_6 \tau)) \quad \text{for each } \Qka \in \D
\end{equation*}
by Lemma \ref{17step5}.

Next, we construct new families of cubes $\mathcal{F}_n(\Qka)$ from the families $\mathcal{E}_j(\Qka)$ recursively such that $\mathcal{F}_1(\Qka) := \mathcal{E}_J(\Qka)$ and
\begin{equation} \label{FnQka}
\mathcal{F}_{n+1}(\Qka) := \bigcup_{Q^{k+nJ}_\beta \in \mathcal{F}_n(\Qka)} \mathcal{E}_J(Q^{k+nJ}_\beta), \quad \text{when } n \geq 1.
\end{equation}
The family $\mathcal{F}_n(\Qka)$ consists of cubes in the generation $k+nJ$, and its cubes are close to the boundary of the cube $J$ generations above. Especially, it holds
\begin{equation} \label{EnJFn}
\mathcal{E}_{nJ}(\Qka) \subset \mathcal{F}_n(\Qka) \quad \text{for each } n \geq 1,
\end{equation}
which we can show by induction; The case $n=1$ follows directly from the definition of the family $\mathcal{F}_n(\Qka)$. Then suppose that $\mathcal{E}_{nJ}(\Qka) \subset \mathcal{F}_n(\Qka)$ for some $n \geq 1$, and take $Q^{k+(n+1)J}_\gamma \in \mathcal{E}_{(n+1)J}(\Qka)$. By the definition of the family $\mathcal{E}_j(\Qka)$, we get
\begin{equation*}
\rho(Q^{k+(n+1)J}_\gamma, X \setminus \cl{\Qka}) \leq C_4 \delta^{k+(n+1)J},
\end{equation*}
and thus also
\begin{equation*}
\rho(Q^{k+(n+1)J}_\gamma, X \setminus \cl{Q^{k+nJ}_\beta}) \leq C_4 \delta^{k+(n+1)J},
\end{equation*}
where $\beta$ is determined by the condition $(k+(n+1)J,\gamma) \preceq (k+nJ,\beta) \preceq (k,\alpha)$. This means that $Q^{k+(n+1)J}_\gamma \in \mathcal{E}_J(Q^{k+nJ}_\beta)$. On the other hand, $Q^{k+nJ}_\beta \in \mathcal{E}_{nJ}(\Qka)$ since
\begin{eqnarray*}
\rho(Q^{k+nJ}_\beta, X \setminus \cl{\Qka}) & \leq & \rho(Q^{k+(n+1)J}_\gamma, X \setminus \cl{\Qka}) \\
& \leq & C_4 \delta^{k+(n+1)J} \\
& \leq & C_4 \delta^{k+nJ}.
\end{eqnarray*}
By combining the induction assumption $\mathcal{E}_{nJ}(\Qka) \subset \mathcal{F}_n(\Qka)$ with the results derived above, we get
\begin{equation*}
\begin{split}
Q^{k+(n+1)J}_\gamma \in \mathcal{E}_J(Q^{k+nJ}_\beta) & \subset \bigcup_{Q^{k+nJ}_\beta \in \mathcal{E}_{nJ}(\Qka)} \mathcal{E}_J(Q^{k+nJ}_\beta) \\
& \subset \bigcup_{Q^{k+nJ}_\beta \in \mathcal{F}_n(\Qka)} \mathcal{E}_J(Q^{k+nJ}_\beta) = \mathcal{F}_{n+1}(\Qka).
\end{split}
\end{equation*}
Hence $\mathcal{E}_{(n+1)J}(\Qka) \subset \mathcal{F}_{n+1}(\Qka)$.

Because of the result (\ref{EnJFn}), also $E_{nJ}(\Qka) \subset F_n(\Qka)$, where $F_n(\Qka)$ is the corresponding set of points
\begin{equation*}
F_n(\Qka) := \{x : x \in Q^{k+nJ}_\beta \text{ for some } Q^{k+nJ}_\beta \in \mathcal{F}_n(\Qka)\}.
\end{equation*}
Because the cubes $Q^{k+nJ}_\beta$ are disjoint by Lemma \ref{QkaQkb} and $E_J(Q^{k+nJ}_\beta) \subset Q^{k+nJ}_\beta$ by their definition, the recursion (\ref{FnQka}) for $\mathcal{F}_n(\Qka)$ also offers a recursion for the measures of the sets $F_n(\Qka)$: $\mu(F_1(\Qka)) = \mu(E_J(\Qka))$ and
\begin{equation} \label{muF1}
\mu(F_{n+1}(\Qka)) = \sum_{Q^{k+nJ}_\beta \in \mathcal{F}_n(\Qka)} \mu(E_J(Q^{k+nJ}_\beta)), \quad \text{when } n \geq 1.
\end{equation}
By applying the recursion (\ref{muF1}) and the inequality (\ref{largeJ}) $n$ times and noticing that the cubes in each $\mathcal{F}_i(\Qka)$ are pairwise disjoint, we get
\begin{eqnarray*}
\mu(F_n(\Qka)) & = & \sum_{Q^{k+(n-1)J}_\beta \in \mathcal{F}_{n-1}(\Qka)} \mu(E_J(Q^{k+(n-1)J}_\beta)) \\
& \leq &  \sum_{Q^{k+(n-1)J}_\beta \in \mathcal{F}_{n-1}(\Qka)} \frac{1}{2} \mu(Q^{k+(n-1)J}_\beta) \\
& = & \frac{1}{2} \mu(F_{n-1}(\Qka)) \\
& \vdots & \\
& \leq & \big( \frac{1}{2} \big)^{n-1} \mu(F_1(\Qka)) \\
& = & \big( \frac{1}{2} \big)^{n-1} \mu(E_J(\Qka)) \\
& \leq & \big( \frac{1}{2} \big)^{n-1} \frac{1}{2} \mu(\Qka) \\
& = & \frac{1}{2^n} \mu(\Qka).
\end{eqnarray*}
Combining this with the result (\ref{EnJFn}) and monotonicity of $\mu$ implies
\begin{equation} \label{mEnJ}
\mu(E_{nJ}(\Qka)) \leq \mu(F_n(\Qka)) \leq \frac{1}{2^n} \mu(\Qka) = \delta^{\eta nJ} \mu(\Qka) \quad \text{for each } n \geq 0,
\end{equation}
where we have chosen $\eta$ such that $\delta^{\eta J} = \frac{1}{2}$. Next fix an index $j \geq 0$ and represent it as $j = nJ+m$, where $n \geq 0$ and $0 \leq m \leq J-1$. Then
\begin{equation*}
\begin{split}
\mu(E_j(\Qka)) & \leq \mu(E_{nJ}(\Qka)) \leq \delta^{\eta nJ} \mu(\Qka) = \delta^{-\eta m} \delta^{\eta j} \mu(\Qka) \\
& \leq \delta^{-\eta J} \delta^{\eta j} \mu(\Qka) = C \delta^{\eta j} \mu(\Qka)
\end{split}
\end{equation*}
by the fact $E_j(\Qka) \subset E_{nJ}(\Qka)$ and the result (\ref{mEnJ}). This proves the claim (\ref{Qboundary2}), i.e. the property (\ref{Qexist}\ref{Qboundary}) of dyadic cubes for small $t$.

In the case when $t$ is large, the property (\ref{Qexist}\ref{Qboundary}) of dyadic cubes, i.e.
\begin{equation*}
\mu(E^k_\alpha(t)) \leq C_2 t^\eta \mu(\Qka) \quad \text{for each } t > C_5,
\end{equation*}
follows in a straightforward way: because $E^k_\alpha(t) \subset \Qka$, it suffices to choose $C_2$ such that $1 \leq C_2 t^\eta$ when $t > C_5$, which is achieved by choosing $C_2 \geq C_5^{-\eta}$. Hence the whole claim (\ref{Qexist}\ref{Qboundary}) has been proved, which also completes the proof of Theorem \ref{Qexist}.

In the end of the section, we prove some additional useful properties for dyadic cubes following from Theorem \ref{Qexist}. First of them tells that large cubes are equal to the whole space if the space is bounded.
\begin{corollary} \label{Xbounded}
The space $(X,\rho)$ is bounded if and only if there exists an index pair $(k,\alpha)$ such that $X = \Qlb$ for each $(l,\beta) \succeq (k,\alpha)$.
\end{corollary}
\begin{proof}
Suppose that $(X,\rho)$ is bounded, i.e. $\diam(X) < \infty$. Choose an index $k \in \Z$ such that $a_0 \delta^k > \diam(X)$, and let $\alpha \in I_k$. Then for any $x \in X$ it holds 
\begin{equation*}
\rho(\zka, x) \leq \diam(X) < a_0 \delta^k,
\end{equation*}
i.e. $X \subset B(\zka, a_0 \delta^k)$. Because $B(\zka, a_0 \delta^k) \subset \Qka$ by the property (\ref{Qexist}\ref{Qball}) of dyadic cubes, we deduce that $X \subset \Qka$. Further, the definition of dyadic cubes implies $X \subset \Qlb$ for each $(l,\beta) \succeq (k,\alpha)$. On the other hand, because cubes are subsets of $X$, the equation holds.

Then suppose that there exists $(k,\alpha)$ such that $X = \Qlb$ for each $(l,\beta) \succeq (k,\alpha)$, which especially implies $X = \Qka$. Because $\Qka \subset B(\zka, C_1 \delta^k)$ by the property (\ref{Qexist}\ref{Qdiam}) of dyadic cubes, we also have $X \subset B(\zka, C_1 \delta^k)$. Hence the space $(X,\rho)$ is bounded.
\end{proof}

Dyadic cubes have a doubling property telling that there is an upper bound for the ratio of the measures of a dyadic cube and its child.
\begin{corollary} \label{Qmeasureparent}
Let $\Qka, Q^{k-1}_\beta \in \D$ such that $\Qka \subset Q^{k-1}_\beta$. Then there exists a constant $C$ independent of $k$, $\alpha$ and $\beta$ such that
\begin{equation*}
\mu(Q^{k-1}_\beta) \leq C \mu(\Qka).
\end{equation*}
\end{corollary}
\begin{proof}
We show first that balls surrounding the cubes $\Qka, Q^{k-1}_\beta$ have the property
\begin{equation} \label{ballparent}
B(z^{k-1}_\beta, C_1 \delta^{k-1}) \subset B(\zka, A_0 \delta^{-1} (1+C_1) \delta^k),
\end{equation}
where $C_1$ is the constant in the property (\ref{Qexist}\ref{Qdiam}) of dyadic cubes. Suppose that $x \in B(z^{k-1}_\beta, C_1 \delta^{k-1})$. Then estimating with the triangle inequality (\ref{quasitriangle}) and the property (\ref{potree}\ref{pochildren}) of the partial order $\preceq$ leads to
\begin{eqnarray*}
\rho(\zka, x) & \leq & A_0 \rho(\zka, z^{k-1}_\beta) + A_0 \rho(z^{k-1}_\beta, x) \\
& < & A_0 \delta^{k-1} + A_0 C_1 \delta^{k-1} \\
& = & A_0 \delta^{-1} (1+C_1) \delta^k,
\end{eqnarray*}
i.e. $x \in B(\zka, A_0 \delta^{-1} (1+C_1) \delta^k)$.

Next, by estimating with the properties (\ref{Qexist}\ref{Qdiam}), (\ref{Qexist}\ref{Qball}) of dyadic cubes, the inclusion (\ref{ballparent}) and monotonicity and the doubling propety (\ref{doubling}) of $\mu$, we get
\begin{eqnarray*}
\mu(Q^{k-1}_\beta) & \leq & \mu(B(z^{k-1}_\beta, C_1 \delta^{k-1})) \\
& \leq & \mu \big( B(\zka, A_0 \delta^{-1} (1+C_1) \delta^k) \big) \\
& \leq & C \mu(B(\zka, a_0 \delta^k)) \\
& \leq & C \mu(\Qka),
\end{eqnarray*}
where $C$ is the doubling coefficient determined by the ratio of $A_0 \delta^{-1} (1+C_1)$ and $a_0$, i.e. $C = A_1^d$ where $d \in \N$ is determined by the condition
\begin{equation*}
2^d a_0 \geq A_0 \delta^{-1} (1+C_1).
\end{equation*}
\end{proof}

The last property of dyadic cubes we will show tells that the boundaries of dyadic cubes have measure zero.
\begin{corollary}
For each $\Qka \in \D$ it holds $\mu(\partial \Qka) = 0$, where $\partial \Qka$ is the boundary of the cube $\Qka$.
\end{corollary}
\begin{proof}
Let us show that
\begin{equation} \label{dQkaQkb}
\partial \Qka \cap \Qkb = \varnothing \quad \text{for each } \Qka, \Qkb \in \D_k.
\end{equation}
The case $\alpha = \beta$ is trivial since $\Qka$ is an open set as a dyadic cube and does not contain its boundary. In the case $\alpha \neq \beta$, suppose that there exists $x \in \partial \Qka \cap \Qkb$ for some $k \in \Z$, $\alpha, \beta \in I_k$. Because $x \in \partial \Qka$, we must have $\rho(x, \Qka) = 0$. On the other hand, because $x \in \Qkb$ and $\Qkb$ is an open set as a dyadic cube, it holds $B(x,\eps) \subset \Qkb$ for some $\eps > 0$. Then especially $\rho(x, X \setminus \Qkb) \geq \eps$. Because $\Qka \subset X \setminus \Qkb$ by the property (\ref{Qexist}\ref{Qtree}) of dyadic cubes, we deduce
\begin{equation*}
0 < \eps \leq \rho(x, X \setminus \Qkb) \leq \rho(x, \Qka) = 0,
\end{equation*}
which is a contradiction. The result (\ref{dQkaQkb}) implies
\begin{equation*}
\partial \Qka \subset X \setminus \bigcup_{\beta \in I_k} \Qkb
\end{equation*}
for any $\Qka \in \D_k$, and thus $\mu(\partial \Qka) = 0$ by the property (\ref{Qexist}\ref{Qfull}) of dyadic cubes and monotonicity of the measure $\mu$.
\end{proof}

The construction of dyadic cubes could be modified such that each generation of dyadic cubes would cover the whole space $X$ instead of excluding a set of measure zero. This could be done by including also boundaries in dyadic cubes following a suitable logic. Such modification has been made, for example, in references [\ref{aimar2}] and [\ref{hytonen}]. However, the modification is usually not very useful since the most typical applications of dyadic cubes are related to integral operators, which are not affected by sets of measure zero.

\end{document}